\documentclass[12pt]{amsart}
\usepackage[nobysame,abbrev,alphabetic]{amsrefs}
\usepackage{amssymb}
\usepackage[all]{xy}

\DeclareMathOperator{\Br}{Br}

\DeclareMathOperator{\Cor}{Cor}

\DeclareMathOperator{\Gal}{Gal}
\DeclareMathOperator{\Hom}{Hom}

\DeclareMathOperator{\Ker}{Ker}

\DeclareMathOperator{\Res}{Res}

\DeclareFontFamily{U}{wncy}{}
\DeclareFontShape{U}{wncy}{m}{n}{<->wncyr10}{}
\DeclareSymbolFont{mcy}{U}{wncy}{m}{n}
\DeclareMathSymbol{\Sha}{\mathord}{mcy}{"58}
\DeclareMathSymbol{\sha}{\mathord}{mcy}{"78}

\begin{document}

\newtheorem{thm}{Theorem}[section]
\newtheorem{cor}[thm]{Corollary}
\newtheorem{lem}[thm]{Lemma}
\newtheorem{prop}[thm]{Proposition}
\newtheorem{defin}[thm]{Definition}
\newtheorem{exam}[thm]{Example}
\newtheorem{examples}[thm]{Examples}
\newtheorem{rem}[thm]{Remark}
\newtheorem{case}{\sl Case}
\newtheorem{claim}{Claim}
\newtheorem{prt}{Part}
\newtheorem*{mainthm}{Main Theorem}
\newtheorem*{thmA}{Theorem A}
\newtheorem*{thmB}{Theorem B}
\newtheorem*{thmC}{Theorem C}
\newtheorem*{thmD}{Theorem D}
\newtheorem{question}[thm]{Question}
\newtheorem*{notation}{Notation}
\swapnumbers
\newtheorem{rems}[thm]{Remarks}
\newtheorem*{acknowledgment}{Acknowledgment}

\newtheorem{questions}[thm]{Questions}
\numberwithin{equation}{section}

\newcommand{\Bock}{\mathrm{Bock}}
\newcommand{\Center}{\mathrm{Center}}
\newcommand{\dec}{\mathrm{dec}}
\newcommand{\diam}{\mathrm{diam}}
\newcommand{\dirlim}{\varinjlim}
\newcommand{\discup}{\ \ensuremath{\mathaccent\cdot\cup}}
\newcommand{\divis}{\mathrm{div}}
\newcommand{\gr}{\mathrm{gr}}
\newcommand{\nek}{,\ldots,}
\newcommand{\indec}{{\rm indec}}
\newcommand{\Ind}{\mathrm{Ind}}
\newcommand{\inv}{^{-1}}
\newcommand{\isom}{\cong}
\newcommand{\Massey}{\mathrm{Massey}}
\newcommand{\ndiv}{\hbox{$\,\not|\,$}}
\newcommand{\nil}{\mathrm{nil}}
\newcommand{\pr}{\mathrm{pr}}
\newcommand{\sep}{\mathrm{sep}}
\newcommand{\sh}{\mathrm{sh}}
\newcommand{\SR}{\mathrm{SR}}
\newcommand{\tagg}{^{''}}
\newcommand{\tensor}{\otimes}
\newcommand{\alp}{\alpha}
\newcommand{\gam}{\gamma}
\newcommand{\Gam}{\Gamma}
\newcommand{\del}{\delta}
\newcommand{\Del}{\Delta}
\newcommand{\eps}{\epsilon}
\newcommand{\lam}{\lambda}
\newcommand{\Lam}{\Lambda}
\newcommand{\sig}{\sigma}
\newcommand{\Sig}{\Sigma}
\newcommand{\bfA}{\mathbf{A}}
\newcommand{\bfB}{\mathbf{B}}
\newcommand{\bfC}{\mathbf{C}}
\newcommand{\bfF}{\mathbf{F}}
\newcommand{\bfP}{\mathbf{P}}
\newcommand{\bfQ}{\mathbf{Q}}
\newcommand{\bfR}{\mathbf{R}}
\newcommand{\bfS}{\mathbf{S}}
\newcommand{\bfT}{\mathbf{T}}
\newcommand{\bfZ}{\mathbf{Z}}
\newcommand{\dbA}{\mathbb{A}}
\newcommand{\dbC}{\mathbb{C}}
\newcommand{\dbF}{\mathbb{F}}
\newcommand{\dbN}{\mathbb{N}}
\newcommand{\dbQ}{\mathbb{Q}}
\newcommand{\dbR}{\mathbb{R}}
\newcommand{\dbU}{\mathbb{U}}
\newcommand{\dbV}{\mathbb{V}}
\newcommand{\dbZ}{\mathbb{Z}}
\newcommand{\grf}{\mathfrak{f}}
\newcommand{\gra}{\mathfrak{a}}
\newcommand{\grA}{\mathfrak{A}}
\newcommand{\grB}{\mathfrak{B}}
\newcommand{\grh}{\mathfrak{h}}
\newcommand{\grI}{\mathfrak{I}}
\newcommand{\grL}{\mathfrak{L}}
\newcommand{\grm}{\mathfrak{m}}
\newcommand{\grp}{\mathfrak{p}}
\newcommand{\grq}{\mathfrak{q}}
\newcommand{\grR}{\mathfrak{R}}
\newcommand{\calA}{\mathcal{A}}
\newcommand{\calB}{\mathcal{B}}
\newcommand{\calC}{\mathcal{C}}
\newcommand{\calE}{\mathcal{E}}
\newcommand{\calG}{\mathcal{G}}
\newcommand{\calH}{\mathcal{H}}
\newcommand{\calK}{\mathcal{K}}
\newcommand{\calL}{\mathcal{L}}
\newcommand{\calM}{\mathcal{M}}
\newcommand{\calW}{\mathcal{W}}
\newcommand{\calV}{\mathcal{V}}

\title[3-fold Massey products]{3-fold Massey products in Galois cohomology - The non-prime case}
\author{Ido Efrat}
\address{Earl Katz Family Chair in Pure Mathematics,
Department of Mathematics, Ben-Gurion University of the Negev, P.O.\  Box 653, Be'er-Sheva 8410501, Israel}
\email{efrat@bgu.ac.il}



\maketitle

\begin{abstract}
For $m\geq2$, let $F$ be a field of characteristic prime to $m$ and containing the roots of unity of order $m$, and let $G_F$ be its absolute Galois group.
We show that the 3-fold Massey products $\langle\chi_1,\chi_2,\chi_3\rangle$, with $\chi_1,\chi_2,\chi_3\in H^1(G_F,\dbZ/m)$ and $\chi_1,\chi_3$ $\dbZ/m$-linearly independent, are non-essential.
This was earlier proved for $m$ prime.
Our proof is based on the study of unitriangular representations of $G_F$.
\end{abstract}

A major open problem in modern Galois theory is to characterize the profinite groups which are realizable as the absolute Galois group $G_F=\Gal(F_\sep/F)$ of a field $F$.
Here it is natural to study the cohomological structure of absolute Galois groups, and in particular, the cohomology ring $H^*(G_F,\dbZ/m)$, under the standard assumption that $F$ has characteristic not dividing $m$ and contains the roots of unity of order $m$.
The celebrated Voevodsky--Rost theorem fully describes $H^*(G_F,\dbZ/m)$ as a graded algebra with respect to the cup product.
Namely, it is the quotient of the tensor $\dbZ/m$-algebra over $H^1(G_F,\dbZ/m)$ modulo the ideal generated by all tensors $(a)_F\tensor(1-a)_F$, where $0,1\neq a\in F$ and $(\cdot)_F$ denotes Kummer elements.
This imposes strong restrictions on the group-theoretic structure of $G_F$ (\cite{CheboluEfratMinac12}, \cite{EfratMinac17}).

In recent years there has been extensive research regarding \textsl{external} cohomological operations on $H^*(G_F,\dbZ/m)$, that is, natural maps on this graded algebra defined using the cochain algebra $C^*(G_F,\dbZ/m)$, and which go beyond the ring structure of $H^*(G_F,\dbZ/m)$.
A major example is the \textsl{$n$-fold Massey product} $\langle\cdot\nek\cdot\rangle\colon H^1(G_F,\dbZ/m)^n\to H^2(G_F,\dbZ/m)$, where $n\geq2$.
This operation can be defined in the cohomology algebra of any differential graded algebra $\calA_\bullet=(\bigoplus_{i\geq 0}\calA_i,d)$, and is a \textsl{multi-valued map}, i.e., for cohomology classes $[a_1]\nek [a_n]\in H^1(\calA_\bullet)$, the Massey product $\langle [a_1]\nek [a_n]\rangle$ is a \textsl{subset} of $H^2(\calA_\bullet)$.
For $n=2$ one has $\langle [a_1],[a_2]\rangle=\{[a_1][a_2]\}$, so in our case the 2-fold Massey product just recovers the cup product.

Thus the first nontrivial case is $n=3$.
We recall that the 3-fold Massey product $\langle [a_1],[a_2],[a_3]\rangle$, where $a_1,a_2,a_3\in \calA_1$ are 1-cocycles satisfying $[a_1][a_2]=0=[a_2][a_3]$, consists of the cohomology classes of the 2-cocycles $a_1 a_{12}+a_{23}a_3$, where $a_{12},a_{23}\in\calA_1$ are chosen so that $a_1a_2=da_{12}$ and $a_2 a_3=da_{23}$.
When $[a_1][a_2]$, $[a_2][a_3]$ do not both vanish, one sets $\langle [a_1],[a_2],[a_3]\rangle=\emptyset$  \cite{LodayVallette12}*{\S9.4.5}.
We refer to e.g., \cite{EfratMatzri17}*{\S1}, \cite{Fenn83}, \cite{Morishita04} for the definition for general $n\geq2$, and to \cite{Kraines66}, \cite{May69}, \cite{Deninger95} for the definition in higher degrees (note that the various sources use different sign conventions).
These latter cases will not be needed in the current paper.

We say that a Massey product $\langle [a_1]\nek[a_n]\rangle$ is \textsl{essential}, if it is nonempty, but does not contain zero.
A prototype of an essential 3-fold Massey product is given by the \textsl{Borromean rings} in algebraic topology \cite{Hillman12}*{\S10.1}.
By contrast,  Deligne, Griffiths, Morgan and Sullivan \cite{DeligneGriffithsMorganSullivan75} prove that compact K\"ahler manifolds are formal in the de Rahm context, which implies that all $n$-fold Massey products for $n\geq3$ are non-essential in this case (see \cite{Huybrechts05}*{Ch.\ 3.A}).

Back to the absolute Galois group case, it was shown in a series of recent works that, for $m=p$ prime, the 3-fold Massey products in $H^*(G_F,\dbZ/p)$ are not essential.
Specifically, when $m=2$ and $F$ is either a global field or a local field, this was shown in a pioneering work by Hopkins and Wickelgren \cite{HopkinsWickelgren15}.
Min\'a\v c and T\^an extended this to the case where $m=2$ and $F$ is an arbitrary field \cite{MinacTan17}, as well as to the case where $m=p$ is an arbitrary prime and $F$ is global or local \cite{MinacTan15b}.
Alternative proofs of the latter two results were given in \cite{EfratMatzri15}.
In fact, it was noted in \cite{EfratMatzri15} that the result for $m=2$ is a cohomological reformulation of a classical characterization of bi-quaternionic algebras, due to Albert \cite{Albert39}.
Other situations where $m=p$ is prime are studied in \cite{MinacTan15a}.
The prime case $m=p$ (for $n=3$) was finally settled by Matzri \cite{Matzri14}, using methods from the theory of central simple algebras and crossed products.
A simplified proof was given in \cite{EfratMatzri17} based solely on Galois cohomology tools.
Almost at the same time as \cite{EfratMatzri17}, Min\'a\v c and T\^an \cite{MinacTan17} gave another proof of the same result, and an additional proof was recently given in  \cite{LamLiuSharifiWangWake20}.
Furthermore, it was shown in  \cite{MinacTan17}*{Example 7.2} how this structural result can actually rule out specific profinite groups from being absolute Galois groups.

The current paper goes beyond the prime case, and extends the above results to arbitrary integers $m\geq2$.

\begin{mainthm}
Suppose that the field $F$ has characteristic not dividing $m$ and that it contains the roots of unity of order $m$.
Let $\chi_1,\chi_2,\chi_3\in H^1(G_F,\dbZ/m)$, where $\chi_1,\chi_3$ are $\dbZ/m$-linearly independent.
Then the 3-fold Massey product $\langle\chi_1,\chi_2,\chi_3\rangle$ is not essential.
\end{mainthm}
We note that in the prime case $m=p$, the independence assumption as well as the assumption about the roots of unity can be easily removed (see \cite{EfratMatzri17}, \cite{MinacTan16}).
On the other hand, Harpaz and Wittenberg show in \cite{HarpazWittenberg19} that the above property of 3-fold Massey products does not hold for $F=\dbQ$ and $m=8$, so the assumption about roots of unity cannot be removed in general.
Following  \cite{GuillotMinacTopazWittenberg18}, \cite{GuillotMinac17}, they prove the above fact for general  $n$-fold Massey products, $n\geq3$, in the case where $F$ is a number field and $m=p$ is prime.
We refer to  \cite{PalSchlank16} for related results.

The proof of the Main Theorem is based on ideas from \cite{EfratMatzri17} and \cite{Matzri14}.
However, we use the different viewpoint of \textsl{unitriangular representations} of $G_F$, that is, continuous homomorphisms into groups of unipotent upper-triangular matrices.
This viewpoint originates from an observation by Dwyer \cite{Dwyer75}, that in the special case of group cohomology, the homological definition of Massey products is equivalent to a construction in terms of unitriangular representations.
More specifically, the elements of a Massey products are pullbacks along such representations of certain cohomology elements constructed using groups of unitriangular matrices (see \S\ref{section on pullbacks}).
This alternative perspective does not only lead to the above generalization to the non-prime case, but also gives a simplified, and much more conceptual approach to the previous results.
E.g., we no longer need calculations with twisted norms, as in \cite{Matzri14}, \cite{EfratMatzri17}, \cite{MinacTan17}.

Furthermore, recent works indicate that Massey products should be seen as the extreme end of a \textsl{spectrum} of unitriangular cohomology elements, which ought to be studied as a whole (the other end consisting of Bockstein elements).
For instance, when $G$ is a free profinite group on a basis $X$, and $G^{(n)}$ is  the $n$th term in its lower $p$-central filtration, this spectrum gives an explicit combinatorial description of $H^2(G/G^{(n)},\dbZ/p)$ in terms of the shuffle algebra on $X$ (\cite{Efrat17}, \cite{Efrat20a}).
An analogous description for the $p$-Zassenhaus filtration is given in \cite{Efrat20b}.
With an eye towards a possible generalization of the Main Theorem to other elements of this spectrum, we develop here a basic ``unitriangular toolkit" which may be helpful to handle these new cohomology elements.

\medskip

We complete this introduction by sketching the main steps of the proof:
We write $\dbU_n(\dbZ/m)$ (resp., $\bar\dbU_n(\dbZ/m)$) for the group of all $(n+1)\times(n+1)$-unitriangular matrices over the ring $\dbZ/m$ (resp., with the $(1,n+1)$-entry omitted).
There is a central extension
\[
0\to\dbZ/m\to\dbU_3(\dbZ/m)\to\bar\dbU_3(\dbZ/m)\to1.
\]

Now assuming that the $3$-fold Massey product $\langle\chi_1,\chi_2,\chi_3\rangle$ is nonempty, there is
a continuous homomorphism $\bar\rho\colon G_F\to\bar\dbU_3(\dbZ/m)$ such that $\bar\rho_{i,i+1}=\chi_i$, $i=1,2,3$.
Let $M_1=\Ker(\chi_1)$.
Using Hilbert theorem 90, we find $\omega\in H^1(M_1,\dbZ/m)$ with the following two properties:
\begin{enumerate}
\item[(i)]
There is a continuous homomorphism
$\begin{bmatrix}
1&\chi_1&\bar\rho'_{13}\\
0&1&\chi_2\\
0&0&1
\end{bmatrix}
\colon G_F\to\dbU_2(\dbZ/m)$ such that $\omega=\bar\rho'_{13}|_{M_1}$; and
\item[(ii)]
There is a continuous homomorphism
$\begin{bmatrix}1&\omega&*\\
0&1&\chi_3|_{M_1}\\
0&0&1\end{bmatrix}\colon M_1\to\dbU_2(\dbZ/m)$.
\end{enumerate}
Using the corestriction--cup product-restriction exact sequence in Galois cohomology, we obtain $\lam_{34}\in H^1(G_F,\dbZ/m)$ such that the modified homomorphism
\[
\begin{bmatrix}
1&\chi_1&\bar\rho'_{13}& \\
0&1&\chi_2&\bar\rho_{24}+\lam_{34}\\
0&0&1&\chi_3\\
0&0&0&1
\end{bmatrix}
\colon G_F\to\bar\dbU_3(\dbZ/m)
\]
can be completed into a continuous homomorphism $G_F\to\dbU_3(\dbZ/m)$.
This implies that the pull-back of the above central extension along the modified homomorphism is $0\in H^2(G_F,\dbZ/m)$.
On the other hand, in view of the structure of the off-diagonal in this matrix,  the pull-back belongs to $\langle\chi_1,\chi_2,\chi_3\rangle$, as desired.

As in \cite{EfratMatzri17}, we prove the Main Theorem more generally in the formal setting of \textsl{$m$-Kummer formations}, which axiomatize the cohomological properties of absolute Galois groups needed in this argument.

\section{Unitriangular Matrices}
\label{section on unitriangular matrices}
Let $n$ be a positive integer and $R$ a commutative unital ring.
The set $\dbU_n(R)$ of all unitriangular (i.e., unipotent and upper-triangular) $(n+1)\times(n+1)$-matrices over $R$ forms a group in the standard way.
The additive group $R^+$ of $R$ embeds as a central subgroup of $\dbU_n(R)$ via the map $a\mapsto I+aE_{1,n+1}$, where $I$ is the unit matrix and $E_{ij}$ denotes the matrix with $1$ at entry $(i,j)$, and $0$ elsewhere.
Let
\[
\bar\dbU_n(R)=\dbU_n(R)/R^+.
\]
We may consider its elements as unitriangular $(n+1)\times(n+1)$-matrices, with the $(1,n+1)$-entry omitted.
Then the multiplication in $\bar\dbU$ is carried out similarly to the usual multiplication of upper-triangular matrices.
We abbreviate $\dbU=\dbU_n(R)$ and $\bar\dbU=\bar\dbU_n(R)$, and let $\pi\colon \dbU\to\bar\dbU$ be the projection map.
Thus there is a central extension
\begin{equation}
\label{central extension}
0\to R^+\to\dbU\xrightarrow{\pi}\bar\dbU\to1.
\end{equation}

For $1\leq i<j\leq n+1$, let $\pr_{ij}\colon\dbU\to R$ be the projection on the $(i,j)$-entry.
When $(i,j)\neq(1,n+1)$ it induces a projection map $\pr_{ij}\colon\bar\dbU\to R$.
We note that $\pr_{i,i+1}\colon\dbU\to R^+$, $i=1,2\nek n$, are group homomorphisms.

Given $1\leq a\leq b\leq n+1$ with $r=b-a$, the restriction to rows $a\leq i\leq b$ and columns $a\leq j\leq b$ gives a homomorphism $\dbU\to\dbU_r(R)$.

From now on we further assume that the ring $R$ is finite.
Let $\bar\dbU$ act on $R^+$ trivially, and consider the cohomology groups $H^l(\bar\dbU,R^+)$, $l=1,2$.
The projections $\pr_{i,i+1}$, $i=1,2\nek n$, may be viewed as elements of $H^1(\bar\dbU,R^+)=\Hom(\bar\dbU,R^+)$.
The extension (\ref{central extension}) corresponds under the Schreier correspondence to an element $\alp=\alp_{n,R}$ of $H^2(\bar\dbU,R^+)$ with the trivial action \cite{NeukirchSchmidtWingberg}*{Th.\ 1.2.4}.

Consider the $2$-cochain
\[
c=\sum_{k=2}^n\pr_{1k}\cup\pr_{k,n+1}\colon \bar\dbU\times\bar\dbU\to R^+, \quad
((m_{ij}),(m'_{ij}))\mapsto\sum_{k=2}^nm_{1k}m'_{k,n+1}.
\]
It is straightforward to verify that it is a 2-cocycle.

\begin{prop}
\label{decomposition of alpha}
The cohomology class of $c$ is $\alp$.
\end{prop}
\begin{proof}
Since $c$ is a $2$-cocycle, $R^+\times \bar\dbU$ is a group with respect to the multiplication
\[
(r,(m_{ij}))*(r',(m'_{ij}))=(r+r'+c((m_{ij}),(m'_{ij})),(m_{ij})(m'_{ij})).
\]
We obtain a central extension
\begin{equation}
\label{central extension 2}
0\to R^+\to R^+\times \bar\dbU\to \bar\dbU\to1,
\end{equation}
which corresponds to the cohomology class of $c$ in $H^2(\dbU,R^+)$ (see the proof of \cite{NeukirchSchmidtWingberg}*{Th.\ 1.2.4}).

It is therefore enough to show that the central extensions (\ref{central extension}) and (\ref{central extension 2}) are equivalent.
To this end consider the map $h=\pr_{1,n+1}\times\pi\colon\dbU\to R^+\times\bar\dbU$.
It is bijective, and commutes with the epimorphisms onto $\bar \dbU$ in both extensions.
It remains to show that $h$ is a group homomorphism.
Indeed, for $(m_{ij}),(m'_{ij})\in\dbU$ we have
\[
\sum_{k=1}^{n+1}m_{1k}m'_{k,n+1}
=m_{1,n+1}+m'_{1,n+1}+c(\pi((m_{ij})),\pi((m'_{ij}))),
\]
which implies that
\[
\begin{split}
&h((m_{ij})(m'_{ij}))=\Bigl(\sum_{k=1}^{n+1}m_{1k}m'_{k,n+1},\pi((m_{ij})(m'_{ij}))\Bigr)\\
&=(m_{1,n+1},\pi((m_{ij})))*(m'_{1,n+1},\pi((m'_{ij})))
=h((m_{ij}))*h((m'_{ij})),
\end{split}
\]
as desired.
\end{proof}

\section{pullbacks}
\label{section on pullbacks}
Let $G$ be a profinite group acting trivially on $R^+$.
Given a continuous homomorphism $\bar\rho\colon G\to\bar\dbU$, we abbreviate $\bar\rho_{ij}=\pr_{ij}\circ\bar\rho$.
Then $\bar\rho_{i,i+1}\colon G\to R^+$, $i=1,2\nek n$, are homomorphisms.

We say that $\bar\rho$ is \textsl{central} if its image is contained in $\Center(\bar\dbU)$.
If $\bar\rho,\bar\rho'\colon G\to \bar\dbU$ are continuous homomorphisms with $\bar\rho'$ central, then  the product map $\bar\rho\cdot\bar\rho'\colon G\to\bar\dbU$, $g\mapsto\bar\rho(g)\bar\rho'(g)$, is also a continuous homomorphism.

Given a continuous homomorphism $\bar\rho\colon G\to\bar\dbU$, let
\[
\bar\rho^*\colon H^2(\bar\dbU,R^+)\to H^2(G,R^+)
\]
be the functorially  induced homomorphism.
Thus $\bar\rho^*\alp$ is the pullback of $\alp=\alp_{n,R}$ to $H^2(G,R^+)$.
It corresponds to the central extension
\[
\label{embedding problem}
0\to R^+\to \dbU\times_{\bar\dbU}G\to G\to1,
\]
where $\dbU\times_{\bar\dbU}G$ is the fiber product of $\dbU$ and $R$ with respect to the projection  $\pi\colon \dbU\to\bar \dbU$ and to $\bar\rho$.
We have $\bar\rho^*\alp=0$ if and only if the latter extension splits, or alternatively, there is a continuous homomorphism $\rho\colon G\to\dbU$ making the following diagram commutative \cite{Hoechsmann68}*{1.1}:
\begin{equation}
\label{embedding problem}
\xymatrix{&&&G\ar[d]^{\bar\rho}\ar@{-->}[dl]_{\rho}&\\
0\ar[r]& R^+\ar[r]& \dbU\ar[r] &\bar\dbU\ar[r]&1.
}
\end{equation}

If $\bar\dbV$ is a subgroup of $\bar\dbU$ and $M=\bar\rho\inv(\bar\dbV)$, then there is a commutative square
\begin{equation}
\label{res commutes with pullbacks}
\xymatrix{
H^2(\bar\dbU,R^+)\ar[r]^{\Res_{\bar\dbV}}\ar[d]_{\bar\rho^*} & H^2(\bar\dbV,R^+)\ar[d]^{(\Res_M\bar\rho)^*} \\
H^2(G,R^+)\ar[r]^{\Res_M}& H^2(M,R^+).
}
\end{equation}

Dwyer \cite{Dwyer75} shows in the discrete context the following connection with Massey products  (see e.g., \cite{Efrat14}*{\S8}  for the profinite context):
Given $\chi_1\nek\chi_n\in H^1(G,R^+)$, where $n\geq2$, the Massey product $\langle\chi_1\nek\chi_n\rangle\subseteq H^2(G,R^+)$ consists of all pullbacks $\bar\rho^*\alp$  with $\bar\rho\colon G\to\bar\dbU$ a continuous homomorphism such that $\chi_i=\bar\rho_{i,i+1}$, $i=1,2\nek n$.
In particular:
\begin{enumerate}
\item[(i)]
$\langle\chi_1\nek\chi_n\rangle\neq\emptyset$ if and only if there is a continuous homomorphism $\bar\rho\colon G\to\bar\dbU$ with $\chi_i=\bar\rho_{i,i+1}$, $i=1,2\nek n$;
\item[(ii)]
$0\in \langle\chi_1\nek\chi_n\rangle$ if and only if there is a continuous homomorphism $\bar\rho\colon G\to\bar\dbU$ with $\chi_i=\bar\rho_{i,i+1}$, $i=1,2\nek n$, and $\bar\rho^*\alp=0$;
i.e., there is a continuous homomorphism $\rho\colon G\to\dbU$ with $\chi_i=\rho_{i,i+1}$, $i=1,2\nek n$.
\end{enumerate}

\begin{rem}
\label{remark on cup products}
\rm
If all of $\bar\rho_{1k}$, $\bar\rho_{k,n+1}$, $k=2\nek n$, are homomorphisms, then they can be identified as elements of $H^1(\bar\dbU,R^+)$.
Hence Proposition \ref{decomposition of alpha} gives rise to a  decomposition $\bar\rho^*\alp=\sum_{k=2}^n\bar\rho_{1k}\cup\bar\rho_{k,n+1}$ in $H^2(G,R^+)$.

For example, for $n=2$, the maps $\bar\rho_{12},\bar\rho_{23}\colon G\to R^+$ are homomorphisms, so $\bar\rho^*\alp=\bar\rho_{12}\cup\bar\rho_{23}$.
More generally, one has:
\end{rem}

\begin{lem}
\label{vanishing of cup products}
Suppose that $n\geq3$ and let $\bar\rho\colon G\to\bar\dbU$ be a continuous homomorphism.
Then $\bar\rho_{i-1,i}\cup\bar\rho_{i,i+1}=0$ in $H^2(G,R^+)$, $i=2\nek n$.
\end{lem}
\begin{proof}
Let $2\leq i \leq n$.
Recall that the projection to rows and columns $i-1,i,i+1$ is a homomorphism $\dbU\to\dbU_2(R)$.
By composing it with $\bar\rho$ we obtain that the maps
\[
\nu=\begin{bmatrix}1&\bar\rho_{i-i,i}&\bar\rho_{i-1,i+1}\\
0&1&\bar\rho_{i,i+1}\\
0&0&1\end{bmatrix}
\colon G\to\dbU_2(R), \
\bar\nu=\begin{bmatrix}1&\bar\rho_{i-i,i}&\\
0&1&\bar\rho_{i,i+1}\\
0&0&1\end{bmatrix}
\colon G\to\bar\dbU_2(R)
\]
are homomorphisms.
Clearly, $\nu$ solves the embedding problem (\ref{embedding problem}) for $n=2$ and $\bar\nu$, so $\bar\nu^*\alp_{2,R}=0$.
By Remark \ref{remark on cup products}, $\bar\rho_{i-1,i}\cup\bar\rho_{i,i+1}=0$.
\end{proof}

As another example, we have the following method to modify the pullbacks $\bar\rho^*\alp$:

\begin{lem}
\label{modified homomorphism}
Suppose that $n\geq3$.
Let $\bar\rho\colon G\to\bar\dbU$ be a continuous homomorphism,  let $\lam,\lam'\colon G\to R^+$ be continuous maps, and let
\[
\bar\rho'=\begin{bmatrix}1&\bar\rho_{12}&\cdots&\bar\rho_{1,n-1}&\bar\rho_{1n}+\lam&\\
&1&\bar\rho_{23}&\cdots&\bar\rho_{2n}&\bar\rho_{2,n+1}+\lam'\\
&&1&\cdots&&\bar\rho_{3n}\\
&&&\ddots&&\vdots\\
&&&&1&\bar\rho_{n-1,n}\\
&&&&&1  \end{bmatrix}
\colon G\to\bar\dbU.
\]
Then $\bar\rho'$ is a homomorphism if and only if $\lam,\lam'$ are homomorphisms.
Moreover, in this case
\[
(\bar\rho')^*\alp=\bar\rho^*\alp+\lam\cup\bar\rho_{n,n+1}+\bar\rho_{12}\cup\lam'.
\]
\end{lem}
\begin{proof}
If $\lam,\lam'$ are homomorphisms, then the maps
\[
I+\lam E_{1n}, I+\lam' E_{2,n+1}\colon G\to\bar\dbU
\]
are central homomorphisms.
By the previous comments,
\[
\bar\rho'=\bar\rho\cdot(I+\lam E_{1n})\cdot(I+\lam' E_{2,n+1})\colon G\to\bar\dbU
\]
is a continuous homomorphism.

Conversely, if  $\bar\rho$, $\bar\rho'$ are homomorphisms, then for $g,h\in G$ we have
\[
\begin{split}
\lam(gh)&=(\bar\rho'(g)\bar\rho'(h))_{1n}-(\bar\rho(g)\bar\rho(h))_{1n}\\
&=\bigl(\bar\rho'_{1n}(g)+\sum_{k=2}^{n-1}\bar\rho_{1k}(g)\bar\rho_{kn}(h)+\bar\rho'_{1n}(h)\bigr)\\
&\quad -\bigl(\bar\rho_{1n}(g)+\sum_{k=2}^{n-1}\bar\rho_{1k}(g)\bar\rho_{kn}(h)+\bar\rho_{1n}(h)\bigr)\\
&=\lam(g)+\lam(h),
\end{split}
\]
and similarly for $\lam'$.

The last assertion follows from Proposition \ref{decomposition of alpha}.
\end{proof}

For $n\geq3$ the projections $\bar\rho_{1k}$, $3\leq k\leq n$, and $\bar\rho_{k,n+1}$, $2\leq k\leq n-1$, need not be homomorphisms, so Proposition \ref{decomposition of alpha} does not yield a cohomological decomposition of $\bar\rho^*\alp$ as a sum of cup products.
Our main result circumvents this difficulty when $n=3$ and $G=G_F$ is the absolute Galois group of a field containing a root of unity of order $m$.
A key argument in the proof will be based on the following analysis:

\begin{exam}
\label{restriction to V}
\rm
Suppose that $n=3$, and let $\dbV$, $\bar\dbV$ be the kernels of $\pr_{12}$, considered as subgroups of $\dbU=\dbU_3(R)$, $\bar\dbU=\bar\dbU_3(R)$, respectively.
The restrictions $\Res_{\bar\dbV}(\pr_{13}),\Res_{\bar\dbV}(\pr_{34})\colon\bar\dbV\to R^+$ are homomorphisms, and may be viewed as elements of $H^1(\bar\dbV,R^+)$.
Since $\Res_{\bar\dbV}(\pr_{12})=0$, the $2$-cocycle $c$ and $\pr_{13}\cup\pr_{34}$ have the same restriction to $\bar\dbV$.

Let $\iota\colon \bar\dbV\to\bar\dbU$, $\bar\iota\colon \bar\dbV\to\bar\dbU$ be the inclusion maps, so $\iota^*=\Res_\dbV$ and $\bar\iota^*=\Res_{\bar\dbV}$.
There is a commutative diagram of central extensions
\[
\xymatrix{
0\ar[r] &R^+\ar@{=}[d]\ar[r]&\dbV\ar[r]\ar@{_{(}->}[d]^{\iota}&\bar\dbV\ar[r]\ar@{_{(}->}[d]^{\bar\iota}&1\\
0\ar[r]&R^+\ar[r]&\dbU\ar[r]^{\pi}&\bar\dbU\ar[r]&1.
}
\]
We have $\dbV=\dbU\times_{\bar\dbU}\bar\dbV$.
By the previous comments, the upper extension in the diagram corresponds to $\Res_{\bar\dbV}(\alp)\in H^2(\bar\dbV,R^+)$, which by
Proposition \ref{decomposition of alpha},  is $\Res_{\bar\dbV}(\pr_{13})\cup\Res_{\bar\dbV}(\pr_{34})$.

Now let $G$ be a profinite group, let $\bar\rho\colon G\to\bar\dbU$ be a continuous homomorphism, and set $M_1=\bar\rho\inv(\bar\dbV)$.
We note that $\Res_{M_1}(\bar\rho_{13})$ and $\Res_{M_1}(\bar\rho_{34})$ are homomorphisms, considered as elements of $H^1(M_1,R^+)$.
By (\ref{res commutes with pullbacks}),
\begin{equation}
\label{Res M1 of pullback}
\begin{split}
\Res_{M_1}(\bar\rho^*\alp)
&=(\Res_{M_1}\bar\rho)^*(\Res_{\bar\dbV}(\alp))\\
&=(\Res_{M_1}\bar\rho)^*(\Res_{\bar\dbV}(\pr_{13}))\cup(\Res_{M_1}\bar\rho)^*(\Res_{\bar\dbV}(\pr_{34})).\\
&=\Res_{M_1}(\bar\rho_{13})\cup\Res_{M_1}(\bar\rho_{34}).
\end{split}
\end{equation}
\end{exam}

\section{Restriction of homomorphisms}
\label{section on restrictions of homomorphisms}
In this section we focus on the case $n=2$, so $\dbU=\dbU_2(R)$.
Let $G$ be a profinite group.

\begin{prop}
\label{psi a homomorphism and commutators}
Suppose that $G=M_1M_2$, where $M_1,M_2$ are closed subgroups of $G$, and $M_1$ is normal in $G$.
Let
\[
\chi_1\colon G\to R^+, \  \chi_2\colon G\to R^+, \ \omega_1\colon M_1\to R^+, \  \omega_2\colon M_2\to R^+
 \]
be continuous homomorphisms such that $\omega_1,\omega_2$ are trivial on $M_1\cap M_2$.
\begin{enumerate}
\item[(a)]
There is a well-defined map $\psi\colon G\to \dbU$, given by
\[
\psi(g)=\begin{bmatrix}
1&\chi_1(g)&\omega_1(g_1)+\omega_2(g_2)\\
0&1&\chi_2(g)\\
0&0&1
\end{bmatrix}
\]
where $g_1\in M_1$, $g_2\in M_2$, and $g=g_1g_2$.
\item[(b)]
The map $\psi$ is a homomorphism if and only if for every $g,g'\in G$ with decompositions $g=g_1g_2$ and $g'=g'_1g'_2$, where $g_1,g'_1\in M_1$ and $g_2,g'_2\in M_2$, one has
\[
\omega_1(g_2g'_1g_2\inv)-\omega_1(g'_1)=\chi_1(g)\chi_2(g').
\]
\end{enumerate}
\end{prop}
\begin{proof}
(a) \quad
Since $\omega_1,\omega_2$ are trivial on $M_1\cap M_2$, the map $g\mapsto \omega_1(g_1)+\omega_2(g_2)$ is well-defined.
Hence $\psi$ is well-defined.

\medskip

(b) \quad
Take $g=g_1g_2$ and $g'=g'_1g'_2$ as in (b).
Then $gg'=g_1(g_2g'_1g_2\inv)g_2g'_2$, where $g_1(g_2g'_1g_2\inv)\in M_1$ and $g_2g'_2\in M_2$.
Hence
\[
\psi(gg')=\begin{bmatrix}
1&\chi_1(gg')&\omega_1(g_1(g_2g'_1g_2\inv))+\omega_2(g_2g'_2)\\
0&1&\chi_2(gg')\\
0&0&1
\end{bmatrix}.
\]
On the other hand
\[
\begin{split}
&\psi(g)\psi(g')=\begin{bmatrix}
1&\chi_1(g)&\omega_1(g_1)+\omega_2(g_2)\\
0&1&\chi_2(g)\\
0&0&1
\end{bmatrix}
\begin{bmatrix}
1&\chi_1(g')&\omega_1(g'_1)+\omega_2(g'_2)\\
0&1&\chi_2(g')\\
0&0&1
\end{bmatrix}\\
&=\begin{bmatrix}
1&\chi_1(g)+\chi_1(g')&\omega_1(g_1)+\omega_2(g_2)+\omega_1(g'_1)+\omega_2(g'_2)
+\chi_1(g)\chi_2(g')\\
0&1&\chi_2(g)+\chi_2(g')\\
0&0&1
\end{bmatrix}.
\end{split}
\]
The assertion now follows from the fact that $\chi_1,\chi_2,\omega_1,\omega_2$ are homomorphisms.
\end{proof}

We now apply this Proposition to the following special case:
Take $R=\dbZ/m$ for an integer $m\geq2$.
Let $\chi_1,\chi_2\in H^1(G,\dbZ/m)$ and $M_1=\Ker(\chi_1)$.
Assume further that $\sig_1\in G$ satisfies $\chi_1(\sig_1)=1$ and $\chi_2(\sig_1)=0$.
Note that $G=M_1\langle\sig_1\rangle={\cdot \hspace{-10pt}\bigcup}_{i=0}^{m-1}M_1\sig_1^i$.

The group $G$ acts on $H^1(M_1,\dbZ/m)$ by $(g\omega)(h)=\omega(ghg\inv)$ for $\omega\in H^1(M_1,\dbZ/m)$, $g\in G$, and $h\in M_1$.
This makes $H^1(M_1,\dbZ/m)$ a module over the group ring $(\dbZ/m)[G]$.

Let $\omega\colon M_1\to \dbZ/m$ be a continuous homomorphism which is trivial on $\sig_1^m$.
Given $z\in \dbZ/m$  we define a \textsl{map} $\psi_z\colon G\to\dbU$ by
\[
\psi_z(h\sig_1^i)=\begin{bmatrix}1&i&\omega(h)+iz\\
0&1&\chi_2(h)\\0&0&1\end{bmatrix},
\]
where $h\in M_1$ and $0\leq i\leq m-1$.

\begin{prop}
\label{rho mu a homomorphism}
The following conditions are equivalent:
\begin{enumerate}
\item[(a)]
There exists a continuous homomorphism $\rho\colon G\to\dbU$ such that $\rho_{12}=\chi_1$, $\rho_{23}=\chi_2$, and $\omega=\Res_{M_1}(\rho_{13})$.
\item[(b)]
The map $\psi_z\colon G\to\dbU$ is a homomorphism for some $z\in \dbZ/m$;
\item[(c)]
The map $\psi_z\colon G\to\dbU$ is a homomorphism for every $z\in \dbZ/m$;
\item[(d)]
$(\sig_1-1)\omega=\Res_{M_1}(\chi_2)$.
\end{enumerate}
Moreover, every homomorphism as in (a) is of the form $\psi_z$ for some $z\in \dbZ/m$.
\end{prop}
\begin{proof}
Consider the previous setup with
\[
M_2=\langle\sig_1\rangle,\   \omega_1=\omega, \  \omega_2(\sig_1^i)=iz.
\]
Note that $\omega_1,\omega_2$ are trivial on $M_1\cap\langle\sig_1\rangle=\langle\sig_1^m\rangle$.

\medskip

(a)$\Rightarrow$(b): \quad
For $\rho$ as in (a) we set  $z=\rho_{13}(\sig_1)\in \dbZ/m$.
Given $g=h\sig_1^i\in G$ with $h\in M_1$ and $0\leq i\leq m-1$, we have
\[
\rho(g)=\rho(h)\rho(\sig_1^i)
=\begin{bmatrix}1&0&\omega(h)\\0&1&\chi_2(h)\\0&0&1\end{bmatrix}
\begin{bmatrix}1&i&iz\\0&1&0\\0&0&1\end{bmatrix}
=\begin{bmatrix}1&i&\omega(h)+iz\\0&1&\chi_2(h)\\0&0&1\end{bmatrix}.
\]
Thus $\rho=\psi_z$.
This proves (b), as well as the last assertion.

\medskip

(c)$\Rightarrow$(a): \quad
We take $\rho=\psi_0$.

\medskip

(b)$\Rightarrow$(d)$\Rightarrow$(c): \quad
Let $z\in\dbZ/m$.
By Proposition \ref{psi a homomorphism and commutators},  $\psi_z$ is a homomorphism if and only if for every $g=h\sig_1^i$ and $g'=h'\sig_1^{i'}$, with $h,h'\in M_1$ and $0\leq i,i'\leq m-1$, one has
\[
\omega(\sig_1^ih'\sig_1^{-i})-\omega(h')=i\chi_2(h').
\]
By induction on $i$, this is equivalent to $\omega(\sig_1h'\sig_1\inv)-\omega(h')=\chi_2(h')$ for every $h'\in M_1$, which is (d).
\end{proof}

\section{Kummer formations}
\label{section on Kummer formations}
We recall from \cite{EfratMatzri17} the following terminology and setup.
Let $G$ be a profinite group and let $A$ be a discrete $G$-module.
For a closed normal subgroup $M$ of $G$ let $A^M$ be the submodule of $G$ fixed by $M$.
There is an induced $G/M$-action on $A^M$.

For open normal subgroups $M\leq M'$ of $G$ let $N_{M'/M}\colon A^M\to A^{M'}$ be the trace homomorphism $a\mapsto \sum_\sig\sig a$, where $\sig$ ranges over a system of representatives for the cosets of $M'$ modulo $M$.

Let $I_{M'/M}$ be the subgroup of $A^M$ consisting of all elements of the form $\sig a-a$ with $\sig\in M'$ and $a\in A^M$.
Note that $N_{M'/M}(I_{M'/M})=\{0\}$.
One sets
\[
\hat H^{-1}(M'/M,A^M)=\Ker(N_{M'/M})/I_{M'/M}
\]
(compare \cite{NeukirchSchmidtWingberg}*{Ch.\ I, \S7}).
When $M'/M$ is cyclic with generator $\sig M$, the subgroup $I_{M'/M}$ consists of all elements $\sig a-a$, with $a\in A^M$ (since $\sig^k-1=(\sig-1)\sum_{i=0}^{k-1}\sig^i$).
Then $\hat H^{-1}(M'/M,A^M)\isom H^1(M'/M,A^M)$  \cite{NeukirchSchmidtWingberg}*{Prop.\ 1.7.1}.

We fix again an integer $m\geq2$.

\begin{defin}
\label{Kummer formations}
\rm
An \textsl{$m$-Kummer formation} $(G,A,\{\kappa_M\}_M)$ consists of a profinite group $G$, a discrete $G$-module $A$, and for each open normal subgroup $M$ of $G$, a $G$-equivariant epimorphism $\kappa_M\colon A^M\to H^1(M,\dbZ/m)$ such that for every such $M$ the following conditions hold:
\begin{enumerate}
\item[(KF1)]
For every of $\chi\in H^1(M,\dbZ/m)$, there is an exact sequence
\[
\begin{split}
H^1(\Ker(\chi),\dbZ/m)\xrightarrow{\Cor_M}H^1(M,\dbZ/m)&\xrightarrow{\cup\chi}H^2(M,\dbZ/m) \\
&\xrightarrow{\Res_{\Ker(\chi)}} H^2(\Ker(\chi),\dbZ/m);
\end{split}
\]
\item[(KF2)]
$\Ker(\kappa_M)=mA^M$;
\item[(KF3)]
For every open normal subgroup $M'$ of $G$ such that $M\leq M'$, there are commutative squares
\[
\xymatrix{
A^M\ar[r]^{\kappa_M\qquad} & H^1(M,\dbZ/m)&&A^M\ar[r]^{\kappa_M\qquad}\ar[d]_{N_{M'/M}} & H^1(M,\dbZ/m)\ar[d]^{\Cor_{M'}} \\
A^{M'}\ar@{^{(}->}[u]\ar[r]^{\kappa_{M'}\qquad} & H^1(M',\dbZ/m)\ar[u]_{\Res_M},&&A^{M'}\ar[r]^{\kappa_{M'}\qquad} & H^1(M',\dbZ/m);
}
\]
\item[(KF4)]
For every open normal subgroup $M'$ of $G$ such that $M\leq M'$ and $M'/M$ is cyclic of order $m$ one has $\hat H^{-1}(M'/M,A^M)=0$.
\end{enumerate}
\end{defin}

\begin{exam}
\label{fields give Kummer formations}
\rm
Extending \cite{EfratMatzri17}*{Example 5.2} (which assumes that $m=p$ is prime), we observe that for a field $F$ of characteristic not dividing $m$ and containing the group $\mu_m$ of $m$th roots of unity,  there is an $m$-Kummer formation $(G_F,F_\sep^\times, \{\kappa_M\}_M)$, where $G_F=\Gal(F_\sep/F)$ is the absolute Galois group of $F$, and for every finite Galois extension $E$ of $F$ with $M=G_E$ the map $\kappa_M\colon E^\times\to H^1(G_E,\mu_m)$ is the \textsl{Kummer homomorphism}.
Thus $\kappa_M\colon H^0(G_E,F_\sep^\times)\to H^1(G_E,\mu_m)$ is the connecting homomorphism arising from the short exact sequence of $G_E$-modules
$1\to\mu_m\to F_\sep^\times\xrightarrow{m}F_\sep^\times\to1$, and is therefore $G_F$-equivariant.
Condition (KF1) corresponds to the isomorphism $E^\times/N_{L/E}L^\times\xrightarrow{\sim}\Br(L/E)$ for a cyclic extension $L/E$ \cite{Draxl83}*{p.\ 73}.
Condition (KF2) follows from the exact sequence $E^\times\xrightarrow{m}E^\times\xrightarrow{\kappa_M}H^1(G_E,\mu_m)$.
Condition (KF3) follows from the commutativity of connecting homomorphisms with restrictions and corestrictions.
For (KF4) use the isomorphism $\hat H^{-1}(M'/M,A^M)\isom H^1(M'/M,A^M)$ for $M'/M$ cyclic (noted above) and Hilbert's Theorem 90.
\end{exam}

\begin{rems}
\label{norms commute with inclusions}
\rm
(1) \quad
Let $M_1,M_2$ be open subgroups of $G$.
Then the following square commutes:
\[
\xymatrix{
A^{M_1}\ar@{^{(}->}[r]\ar[d]_{N_{M_1M_2/M_1}}&A^{M_1\cap M_2}\ar[d]^{N_{M_2/M_1\cap M_2}}\\
A^{M_1M_2}\ar@{^{(}->}[r]& A^{M_2}.
}
\]
Indeed, a system of representatives for the cosets of $M_2/(M_1\cap M_2)$ is also a system of representatives for the cosets of $M_1M_2/M_1$.

\medskip

(2) \quad
For normal open subgroups $M\leq M'$ of $G$ and for $\sig\in G$, one has a commutative square
\[
\xymatrix{
A^M\ar[d]_{N_{M'/M}}\ar[r]^{\sig}& A^M\ar[d]^{N_{M'/M}}\\
A^{M'}\ar[r]^{\sig}&A^{M'}.
}
\]
Here we may replace $\sig$ by any element of the group ring $\dbZ[G]$.
\end{rems}

\section{Hilbert 90}
\label{section on Hilbert 90}
Let $(G,A,\{\kappa_M\}_M)$ be an $m$-Kummer formation.
Let $M<M'$ be open subgroups of $G$ with $M'/M$ cyclic of order $m$.
Choose $\sig\in M'$ such that $M'=\langle M,\sig\rangle$.

Let $N_{M'/M}\inv(A^G)$ be the inverse image of $A^G(\subseteq A^{M'})$ under the map $N_{M'/M}\colon A^M\to A^{M'}$.
In view of (KF4), there is a commutative diagram with exact rows
\[
\xymatrix{
 A^M\ar[r]^{\sig-1\quad\  }&N_{M'/M}\inv(A^G)\ar[r]^{\quad\  N_{M'/M}}&A^G&\\
 &A^G\ar[r]^{m}\ar@{^{(}->}[u]&mA^G\ar[r]\ar@{^{(}->}[u]&0.
}
\]
We deduce the following variant of Hilbert theorem 90 for Kummer formations:

\begin{prop}
\label{exact sequence}
in the above setup, there is an induced exact sequence
\[
A^M\xrightarrow{\sig-1} N_{M'/M}\inv(A^G)/A^G\xrightarrow{\bar N_{M'/M}} A^G/mA^G.
\]
\end{prop}

We say that open normal subgroups $M_1,M_3$ of $G$ are \textsl{$m$-independent} if $G_1/M_1\isom G/M_3\isom\dbZ/m$ and the natural map $G/(M_1\cap M_3)\to(G/M_1)\times (G/M_3)$ is an isomorphism
(the unnatural indexing is due to its application in the next section).
For $\chi_1,\chi_3\in H^1(G,\dbZ/m)$, we note that $M_i=\Ker(\chi_i)$, $i=1,3$, satisfy this condition if and only if $\chi_1,\chi_3$ are $\dbZ/m$-linearly independent.

The next proposition is the key technical step in our proof.

\begin{prop}
\label{key prop}
Let $M_1,M_3$ be open normal subgroups of $G$ which are $m$-independent, and set $M=M_1\cap M_3$.
Let $y_i\in A^{M_i}$, $i=1,3$, satisfy $\kappa_G(N_{G/M_1}(y_1))=\kappa_G(N_{G/M_3}(y_3))$.
Let $\sig_1\in M_3$ satisfy $G=\langle M_1,\sig_1\rangle$.
Then there exists $t\in A^M$ such that
\[
(\sig_1-1)N_{M_1/M}t=N_{G/M_1}y_1\pmod{mA^{M_1}}.
\]
\end{prop}
\begin{proof}
We choose  $\sig_3\in M_1$ such that $G=\langle M_3,\sig_3\rangle$, and denote $M'=\langle M,\sig_1\sig_3\rangle$.
Since $M_1,M_3$ are $m$-independent, $(M':M)=m$ and
\[
G=M_1M_3=M_1M'=M_3M', \quad M=M_1\cap M'=M_3\cap M' .
\]
By Remark \ref{norms commute with inclusions}(1), the following diagram commutes:
\[
\xymatrix{
A^{M_1}\ar@{^{(}->}[r]\ar[d]_{N_{G/M_1}}&A^M\ar[d]^{N_{M'/M}}
&A^{M_3}\ar@{_{(}->}[l]\ar[d]^{N_{G/M_3}}\\
A^G\ar@{^{(}->}[r]& A^{M'}&
A^G\ar@{_{(}->}[l].
}
\]
Hence
\begin{equation}
\label{norms and kappa}
N_{M'/M}(y_3-y_1)=N_{G/M_3}y_3-N_{G/M_1}y_1\in A^G\ (\subseteq A^{M_1}),
 \end{equation}
 so $y_3-y_1\in N_{M'/M}\inv(A^G)$, and
\[
\kappa_G(N_{M'/M}(y_3-y_1))=\kappa_G(N_{G/M_3}y_3)-\kappa_G(N_{G/M_1}y_1)=0.
\]
By (KF2),
\begin{equation}
\label{norm in mAG}
N_{M'/M}(y_3-y_1)\in mA^G.
\end{equation}
Therefore Proposition \ref{exact sequence} (with $\sig=\sig_1\sig_3$) yields $t\in A^M$ and $b\in A^G$ such that
\[
(\sig_1\sig_3-1)t=y_3-y_1+b.
\]

As $\sig_3\in M_1$, Remark \ref{norms commute with inclusions} gives commutative squares
\[
\xymatrix{
A^M\ar[r]^{\sig_1\sig_3-1}\ar[d]_{N_{M_1/M}}&A^M\ar[d]^{N_{M_1/M}}&&  A^{M_3}\ar@{^{(}->}[r]\ar[d]_{N_{G/M_3}}& A^M\ar[d]^{N_{M_1/M}}\\
A^{M_1}\ar[r]^{\sig_1-1}&A^{M_1},&&  A^G\ar@{^{(}->}[r]& A^{M_1}.
}
\]
Considering $y_1,y_3$ as elements of $A^M$, we obtain
\[
\begin{split}
&(\sig_1-1)N_{M_1/M}t
=N_{M_1/M}((\sig_1\sig_3-1)t)\\
=&N_{M_1/M}(y_3)-N_{M_1/M}(y_1)+N_{M_1/M}(b)\\
=&N_{G/M_3}(y_3)-my_1+mb\equiv N_{G/M_3}(y_3)
\equiv N_{G/M_1}(y_1)\pmod{mA^{M_1}},
\end{split}
\]
where in the last step we used (\ref{norms and kappa}) and (\ref{norm in mAG}).
\end{proof}

We now interpret the previous proposition in cohomological terms.

\begin{cor}
\label{cor to key prop}
Let $\bar\rho\colon G\to\bar\dbU_3(\dbZ/m)$ be a continuous homomorphism.
Assume that $M_i=\Ker(\bar\rho_{i,i+1})$, $i=1,3$, are $m$-independent, and set $M=M_1\cap M_3$.
Let  $\sig_1\in M_3$ satisfy $G=\langle M_1,\sig_1\rangle$.
Then there exists  $\omega\in H^1(M_1,\dbZ/m) $ such that
\[
\omega\in\Cor_{M_1}H^1(M,\dbZ/m) , \quad (\sig_1-1)\omega=\Res_{M_1}(\bar\rho_{23}).
\]
\end{cor}
\begin{proof}
By Lemma \ref{vanishing of cup products}, $\bar\rho_{12}\cup\bar\rho_{23}=0=\bar\rho_{23}\cup\bar\rho_{34}$.
For both $i=1,3$, (KF1) implies that $\bar\rho_{23}$ is in the image of $\Cor_G\colon H^1(M_i,\dbZ/m)\to H^1(G,\dbZ/m)$.
Hence there exists $y_i\in A^{M_i}$ with $\Cor_G(\kappa_{M_i}(y_i))=\bar\rho_{23}$.
By (KF3), $\kappa_G(N_{G/M_i}y_i)=\bar\rho_{23}$.
Proposition \ref{key prop} yields $t\in A^M$ such that
\[
(\sig_1-1)N_{M_1/M}t\equiv N_{G/M_1}y_1\pmod{mA^{M_1}}.
\]
Let $\omega=\Cor_{M_1}(\kappa_M(t))$.
By (KF3) and Remark \ref{norms commute with inclusions}(2),
\[
\begin{split}
(\sig_1-1)\omega&=(\sig_1-1)\kappa_{M_1}(N_{M_1/M}(t))
=\kappa_{M_1}((\sig_1-1)N_{M_1/M}(t)) \\
&=\kappa_{M_1}(N_{G/M_1}(y_1))
=\Res_{M_1}\kappa_G(N_{G/M_1}(y_1))=\Res_{M_1}(\bar\rho_{23}).
\end{split}
\]
\end{proof}

\section{The main result}
\label{section on the main result}
In this section we take $n=3$, and set as before $R=\dbZ/m$ for $m\geq2$.
Then $\alp$ is the cohomology element corresponding to the sequence
\[
0\to \dbZ/m\to \dbU_3(\dbZ/m)\to\bar\dbU_3(\dbZ/m)\to1.
\]

Let $G$ be a profinite group, and let $\bar\rho\colon G\to\bar\dbU_3(\dbZ/m)$ be a continuous homomorphism.
For $i=1,3$ we set $M_i=\Ker(\bar\rho_{i,i+1})$, and $M=M_1\cap M_3$.
We recall that $\omega=\Res_{M_1}(\bar\rho_{13})\colon M_1\to \dbZ/m$ is a homomorphism, and view it as an element of $H^1(M_1,\dbZ/m)$.

\begin{prop}
\label{equivalence Cor and modified rho}
Assume that (KF1) holds for every open normal subgroup $M$ of $G$.
The following conditions are equivalent:
\begin{enumerate}
\item[(a)]
$\Res_{M_1}(\bar\rho^*\alp)=0$;
\item[(b)]
There exists a continuous homomorphism $\lam\colon G\to \dbZ/m$ such that $(\bar\rho(I-\lam E_{24}))^*\alp=0$;
\item[(c)]
$\omega\in\Cor_{M_1}(H^1(M,\dbZ/m))$.
\end{enumerate}
\end{prop}
\begin{proof}
(a)$\Leftrightarrow$(b): \quad
For every $\lam\in H^1(G,\dbZ/m)$, Lemma \ref{modified homomorphism} shows that
\[
\bar\rho'=\bar\rho(I-\lam E_{24})=\begin{bmatrix}1&\bar\rho_{12}&\bar\rho_{13}&\\0&1&\bar\rho_{23}&\bar\rho_{24}-\lam\\0&0&1&\bar\rho_{34}\\0&0&0&1\end{bmatrix}
\colon G\to\bar\dbU
\]
is a continuous homomorphism, and $(\bar\rho')^*\alp=\bar\rho^*\alp-\bar\rho_{12}\cup\lam$.
Now by (KF1), $\Res_{M_1}(\bar\rho^*\alp)=0$  if and only if $\bar\rho^*\alp=\bar\rho_{12}\cup\lam$ for some $\lam\in H^1(G,\dbZ/m)$.

\medskip

(a)$\Leftrightarrow$(c): \quad
By (\ref{Res M1 of pullback}),
$\Res_{M_1}(\bar\rho^*\alp)=\omega\cup\Res_{M_1}(\bar\rho_{34})$.
We now apply (KF1) for the group $M_1$ and the homomorphism $\Res_{M_1}(\bar\rho_{34})$.
\end{proof}

\medskip

We now prove the Main Theorem in the more general context of $m$-Kummer formations.
The statement in terms of Massey products given in the Introduction follows from the statement below, in view of the interpretation of Massey product elements as pullbacks $\bar\rho^*\alp$, given in \S\ref{section on pullbacks}, and by Example \ref{fields give Kummer formations}.

\begin{thm}
Let $G$ be the underlying group of an $m$-Kummer formation.
Let $\chi_1,\chi_2,\chi_3\in H^1(G,\dbZ/m)$, where $\chi_1,\chi_3$ are $\dbZ/m$-linearly independent.
If there exists a continuous homomorphism $\bar\rho\colon G\to\bar\dbU_3(\dbZ/m)$ with $\chi_i=\bar\rho_{i,i+1}$, $i=1,2,3$, then there exists such a  homomorphism for which $\bar\rho^*\alp=0$.
\end{thm}
\begin{proof}
Let $\bar\rho\colon G\to\bar\dbU_3(\dbZ/m)$ be a continuous homomorphism.
Let $M_1,M_3$ be as above, and choose $\sig_1\in G$ such that $G=M_1\sig_1=\bigcup_{i=0}^{m-1}M_1\sig_1^i$.
Corollary \ref{cor to key prop} yields  $\omega\in H^1(M_1,\dbZ/m)$ such that
\[
\omega\in\Cor_{M_1}(H^1(M,\dbZ/m)), \qquad
(\sig_1-1)\omega=\Res_{M_1}(\chi_2).
\]

Proposition \ref{rho mu a homomorphism} yields a continuous map $\bar\rho'_{13}\colon G\to\dbZ/m$ such that
\[
\begin{bmatrix}
1&\bar\rho_{12}&\bar\rho'_{13}\\
0&1&\bar\rho_{23}\\
0&0&1
\end{bmatrix}
\colon G\to\bar\dbU_2(\dbZ/m)
\]
is a homomorphism, and $\Res_{M_1}(\bar\rho'_{13})=\omega$.
It extends to a continuous homomorphism
\[
\bar\rho'=\begin{bmatrix}
1&\bar\rho_{12}&\bar\rho'_{13}&\\
0&1&\bar\rho_{23}&\bar\rho_{24}\\
0&0&1&\bar\rho_{34}\\
0&0&0&1
\end{bmatrix}
\colon G\to\bar\dbU_3(\dbZ/m).
\]
By Lemma \ref{modified homomorphism}, $\lam_{13}=\bar\rho_{13}-\bar\rho'_{13}\colon G\to \dbZ/m$ is a homomorphism.
Note that $\bar\rho'=\bar\rho(I-\lam_{13}E_{13})$.

Proposition \ref{equivalence Cor and modified rho} yields $\lam_{24}\in H^1(G,\dbZ/m)$ such that $(\bar\rho'(I-\lam_{24}))^*\alp=0$.
Therefore
\[
\bar\rho(I-\lam_{13}E_{13})(I-\lam_{24}E_{24})
=\begin{bmatrix}
1&\bar\rho_{12}&\bar\rho_{13}-\lam_{13}&\\
0&1&\bar\rho_{23}&\bar\rho_{24}-\lam_{24}\\
0&0&1&\bar\rho_{34}\\
0&0&0&1
\end{bmatrix}
\colon G\to\bar\dbU_3(\dbZ/m)
\]
is a homomorphism as required.
\end{proof}

\end{document}